\documentclass[12pt]{article}
\usepackage[utf8]{inputenc}
\usepackage[T1]{fontenc}
\usepackage{epigraph}
\usepackage{mathrsfs}

\usepackage{enumitem}

\usepackage{amsmath,latexsym,amsthm,mathtools}
\usepackage{amssymb}

\usepackage[bookmarks,bookmarksnumbered,pdfstartview=FitBH,backref=page]{hyperref}
\hypersetup{
colorlinks=true, 
linkcolor=blue,
citecolor= blue,
linktoc= all, 
}

\newtheorem{Theorem}{Theorem}[section]
\newtheorem{proposition}[Theorem]{Proposition}
\newtheorem{theorem}[Theorem]{Theorem}
\newtheorem{corollary}[Theorem]{Corollary}
\newtheorem{lemma}[Theorem]{Lemma}

\newtheorem*{theoremintro}{Theorem}

\theoremstyle{definition}

\newtheorem{definition}[Theorem]{Definition}
\newtheorem{remark}[Theorem]{Remark}
\newtheorem*{question*}{Question}

\newtheoremstyle{Fancyplain}
{\topsep}   
{\topsep}   
{}  
{0pt}       
{\bfseries} 
{}         
{5pt plus 1pt minus 1pt} 
{\thmname{#1} \thmnumber{#2}:\thmnote{\ #3}.\ }

\theoremstyle{Fancyplain}
\newtheorem{step}{Step}

\def\defin#1{\textbf{#1}} 

\newcommand\cost{\mathcal{C}}
\DeclareMathOperator{\dom}{\mathrm{dom}}
\DeclareMathOperator{\rng}{\mathrm{rng}}

\newcommand{\N}{\mathbb N}

\newcommand{\Z}{\mathbb Z}
\newcommand{\inv}{^{-1}}
\newcommand{\la}{\left\langle}
\newcommand{\ra}{\right\rangle}
\DeclareMathOperator{\supp}{\mathrm{supp}}
\newcommand{\abs}[1]{\left\lvert #1\right\rvert}

\newcommand{\PK}{\mathcal{K}}

\def\id{\mathrm{id}}
\def\Aut{\mathrm{Aut}}
\def\FF{\mathbf F}
\newcommand{\acts}{\curvearrowright}

\def\RR{\mathscr{R}}

\newcommand{\Sym}{\mathrm{Sym}}
\newcommand\Stab{\mathrm{Stab}}
\newcommand\Sub{\mathrm{Sub}}
\newcommand{\induced}[1]{\updownarrow #1}

\author{A. Carderi, D. Gaboriau, F. Le Maître}

\title{On dense totipotent free subgroups in full groups}
\date{}

\begin{document}
\maketitle
\begin{abstract}
We study probability measure preserving (p.m.p.) non-free actions of free groups and the associated IRS's. The perfect kernel of a countable group $\Gamma$ is the largest closed subspace of the space of subgroups of $\Gamma$ without isolated points.
We introduce the class of totipotent ergodic p.m.p.\ actions of $\Gamma$: 
those for which almost every point-stabilizer has dense conjugacy class 
in the perfect kernel. Equivalently,
 the support of the associated  IRS is as large as possible,
 namely it is equal to the whole perfect kernel.
We prove that every ergodic p.m.p.\ equivalence relation $\RR$ of cost $<r$ can 
be realized by the orbits of an action of the free group $\FF_r$ on $r$ generators that is totipotent and such that the image in the full group $[\RR]$ is 
dense. We explain why these actions have no minimal models.
	This also provides a continuum of pairwise orbit inequivalent invariant random subgroups of $\FF_r$, all of whose supports are equal to the whole space of infinite index subgroups.
We are led to introduce a property of topologically generating pairs for full groups (we call evanescence) and establish a genericity result about their existence. 
We show that their existence characterizes cost $1$. 
\end{abstract}


\noindent
\textbf{keywords:}
{Measurable group actions, non-free actions, free groups, transitive actions of countable groups, IRS, space of subgroups,
ergodic equivalence relations, orbit equivalence.}

\tableofcontents

\section{Introduction}
\epigraph{\it In this context, clarifying precisely what is meant by 
“totipotency” 
and how it is experimentally determined will both avoid unnecessary controversy 
and potentially reduce inappropriate barriers to research. }{--- M.\ Condic, 
	\cite{condicTotipotencyWhatIt2014}}
Let $\Gamma$ be a countable discrete group. 
Denote by $\Sub(\Gamma)$ the space of  subgroups of $\Gamma$. It is equipped with the compact totally disconnected topology of pointwise convergence and with the continuous 
$\Gamma$-action by conjugation. 
Let $\beta$ be a Borel $\Gamma$-action on the standard Borel space $X\simeq 
[0,1]$. 
Its \defin{stabilizer map} 
\newcommand{\Stabmap}{\mathrm{Stab}}
\begin{eqnarray*}
\Stabmap^{\beta}&:&X\to \Sub(\Gamma)\\
&& x\mapsto \{\gamma\in\Gamma\colon \beta(\gamma)x=x\}
\end{eqnarray*}
 is $\Gamma$-equivariant.
 If $\mu$ is a probability measure on $X$ which is preserved by 
 $\beta$, then the push-forward measure $\Stabmap^{\beta}_*\mu$ is invariant under conjugation. It is 
 the prototype of an Invariant Random Subgroup (\defin{IRS}).
When $\mu$ is atomless and the stabilizer map is essentially injective (a.k.a.\ 
the action 
$\beta$ is \defin{totally non-free}),
the support of the associated IRS $\Stabmap^{\beta}_*(\mu)$ has no isolated 
points: it is a perfect set.
The largest closed subspace of $\Sub(\Gamma)$ with no isolated points 
is called the \defin{perfect kernel} of $\Sub(\Gamma)$.
We say that an ergodic
probability measure-preserving 
(\defin{p.m.p.})\
action is \defin{totipotent} when the support of its IRS is equal to the 
perfect kernel of $\Sub(\Gamma)$. By ergodicity, the following stronger 
property holds: {\em almost every element of the associated IRS has dense 
orbit in 
the perfect kernel} (see Proposition \ref{prop: dense orbit}). We call such an 
IRS \textbf{totipotent}.

Given a p.m.p.\ action $\Gamma\acts^{\beta} (X,\mu)$, we consider the 
associated
\defin{p.m.p.\ equivalence relation} $$\RR^{\beta}\coloneqq\{(x,y) \in X\times 
X: 
\beta(\Gamma)x=\beta(\Gamma)y\}$$ and its \defin{full group} $[\RR^{\beta}]$ 
as the group of 
all measure-preserving transformations whose graph is contained in $\RR^\beta$.
The (bi-invariant) \defin{uniform distance} between two measure-preserving 
transformations $S$ and $T$ is defined by 
$d_u(T,S)\coloneqq\mu(\{x\in X: S(x)\neq T(x)\})$. It endows the full group 
$[\RR^{\beta}]$
with a Polish group structure.
The \defin{cost} is a numerical invariant attached to the equivalence relation 
$\RR^{\beta}$. If $\beta$ is a 
p.m.p.\ action of the free group $\FF_r$ on $r$ 
generators, then 
{\em the cost of $\RR^{\beta}$ is exactly $r$ when $\beta$ is free
and the cost of $\RR^{\beta}$ is $<r$ when $\beta$ is non-free}
 \cite{Gab00}.

The main result of \cite{lemaitrenumbertopologicalgenerators2014} is that for 
any 
ergodic p.m.p.\ equivalence relation $\RR $,
{\em if $\RR$ has cost $<r$ for some integer $r\geq 2$, then there 
exists a homomorphism $\tau:\FF_r\to [\RR]$ with dense image}. 

This result has been sharpened in order to ensure that \emph{the homomorphism 
$\tau$ is injective}. Actually,  
the associated (almost 
everywhere defined) p.m.p.\ action $\alpha_\tau$ can be made to satisfy the 
following 
two opposite conditions: \emph{high faithfulness} and  \emph{amenability on 
$\mu$-almost 
	every orbit} \cite{lemaitreHighlyFaithfulActions2018}.

These two conditions can be phrased in terms of the support of the IRS 
associated to the 
action: 
the first one means that the support contains the trivial subgroup, and one 
can show that the 
second one is equivalent to the support containing a 
co-amenable subgroup 
 (which in the construction of 
\cite{lemaitreHighlyFaithfulActions2018} is the kernel 
of a certain surjective homomorphism $\FF_r\to \Z$).

The purpose of the present paper is to show that the homomorphism can  
be chosen so that the support of the associated IRS is actually the largest perfect subspace of $\Sub(\FF_r)$, 
 which consists 
of all its infinite index subgroups (see Proposition~\ref{rem: perfect kernel for free groups}).

\begin{theoremintro}
	Let $\RR $ be an ergodic p.m.p.\ equivalence relation whose cost is 
	$<r$ for some integer $r\geq 2$. Then there exists a homomorphism 
	$\tau: \FF_r\to [\RR ]$ whose 
	image is dense and whose associated p.m.p.\ action $\alpha_\tau$ is totipotent.
\end{theoremintro}

The density in $[\RR]$ of the image of $\tau$ implies that 
$\RR^{\alpha_\tau}\simeq \RR$ and that the stabilizer map 
$\Stabmap^{\alpha_\tau}$ 
is essentially injective \cite[Prop.~2.4]{lemaitreHighlyFaithfulActions2018}.
 In particular, the actions $\FF_r\acts (\Sub(\FF_r),\Stabmap^{\alpha_\tau}_*\mu)$ and $\FF_r\acts^{\alpha_\tau} (X,\mu)$ are conjugate (thus produce the same equivalence relation)
  and almost every subgroup for the IRS $\Stabmap^{\alpha_\tau}_*\mu$ equals its own normalizer. 
It follows that up to isomorphism, \emph{every p.m.p.\ ergodic equivalence 
relation of cost $<r$ comes from a totipotent IRS of $\FF_r$	
(actually, from continuum many different totipotent IRS's of $\FF_r$, see Remark~\ref{rem: continuum many IRS for R})}.

Such a statement is optimal since p.m.p.\ equivalence relations of cost $\geq 
r$ 
cannot come from a non-free $\FF_r$ action. To our knowledge, it was not 
even clear until now whether $\FF_r$ admits ergodic totipotent IRS's. Since 
there are continuum many pairwise non-isomorphic ergodic p.m.p.\ equivalence 
relations of cost $<r$, our approach provides continuum many 
pairwise distinct ergodic totipotent IRS's 
of the 
free group on $r$ generators, whose associated equivalence relations are even 
non-isomorphic.

Another interesting fact about totipotent p.m.p.\ $\FF_r$-actions is that they have no minimal model, i.e., they cannot be realized as minimal actions on a 
compact space. Indeed, 
it follows from a result of Glasner-Weiss \cite[Cor. 
4.3]{glasnerUniformlyRecurrentSubgroups2015} that as 
soon as the support of the IRS of a given 
p.m.p.\ action contains two distinct minimal subsets (e.g.\ when it contains 
two distinct fixed 
points), the action does not admit a minimal model (see Theorem \ref{thm: no 
minimal model}).
In our case the perfect kernel of $\Sub(\FF_r)$ contains a 
continuum of fixed points (namely, all infinite index normal subgroups), so 
that 
totipotent p.m.p.\ actions of $\FF_r$ are actually very far from admitting a 
minimal model.

Let us now recall the context around our construction. 
The term IRS was coined by Abert-Glasner-Virag  
\cite{Abert-Glasner-Virag-2014} and has become an important subject on its own 
at the intersection of group theory, probability theory and dynamical systems.
The notion of IRS is a natural generalization of a normal subgroup, especially 
in the direction of superrigidity type results. It has thus been present 
implicitly in the work of many authors, a famous landmark being the 
Stuck-Zimmer Theorem \cite{Stuck-Zimmer-1994}, which gives examples of groups 
admitting very few IRS's.
On the contrary, some groups admit a ``zoo'' of IRS's, 
starting 
with free groups \cite{bowenInvariantRandomSubgroups2015}  (see 
\cite{BGK-2015-IRS-Lamp,Bowen-Grigor-Krav-2017,
	kechrisCoinductionInvariantRandom2019}  for other 
examples).

In particular, Bowen proved that every p.m.p.\ ergodic equivalence 
relation of cost $<r$ comes from some IRS of $\FF_r$. He obtained this 
result 
	through a Baire category argument which required that the first generator 
	acts freely. In particular, such IRS's can never be totipotent.

Eisenmann and Glasner then used homomorphisms $\FF_r\to[\RR]$ with dense image 
so as to obtain interesting IRS's of $\FF_r$ 
\cite{eisenmannGenericIRSFree2016}. 
They proved that given a homomorphism $\Gamma\to[\RR ]$ with dense image, 
the associated IRS is always  
co-highly transitive almost surely, which means that for almost every $\Lambda\leq 
\Gamma$, the 
$\Gamma$-action on $\Gamma/\Lambda$ is $n$-transitive for every $n\in\N$. They 
also showed that the IRS's of $\FF_r$ obtained by Bowen for cost $1$ 
equivalence relations 
are faithful and moreover almost surely co-amenable.

The third-named author then used 
a modified version of his result on the topological rank of full groups to 
show that every p.m.p.\ ergodic equivalence 
relation of cost $<r$ comes from a co-amenable, co-highly transitive and 
faithful
 IRS of $\FF_r$ \cite{lemaitreHighlyFaithfulActions2018}. Also in this 
 construction, the first 
generator continues to act freely, thus preventing totipotency. Let us now briefly 
explain how our new construction (Section~\ref{sect: Proof of the main theorem}) allows us to circumvent this. 

The main idea is to use a smaller set $Y\subsetneq  X$ such that the 
restriction of $\RR $ to $Y$ still has cost $<r$, so that we can find some 
homomorphism $\FF_r\to [\RR _{\restriction Y}]$ with dense image. 
This provides us with some extra space in order to obtain totipotency via a 
well-chosen perturbation of the above homomorphism. 

This perturbation is  obtained by 
mimicking all Schreier balls on $X\setminus Y$ and then merging these amplifications with the 
action on $Y$ so as to obtain both density  in $[\RR]$ and totipotency.
The use of evanescent pairs of topological generators (see Definition~\ref{def: 
evanescent pair}) 
with Theorem \ref{Th: existence of evanescent pair} and Proposition \ref{prop: 
the good pre-cycles} will grant us that 
this perturbation maintains
the density. We establish in Theorem~\ref{thm: chara evanescent} that the existence of an evanescent pair of topological generators is equivalent with $\RR$ having cost $1$.

Finally, let us mention the case of the free group on infinitely many 
generators $\FF_\infty$. Here, the space of subgroups is already perfect  (see Proposition~\ref{rem: 
perfect kernel for free groups}), and 
one can easily adapt our arguments to show that:
{\em 
	For every ergodic p.m.p.\ equivalence relation $\RR $, there 
	exists a homomorphism $\tau: 
	\FF_\infty\to [\RR ]$ whose 
	image is dense and whose associated p.m.p.\ action $\alpha_\tau$ is totipotent.
}

This result could however also be obtained by a purely Baire-categorical argument: it is not hard to see that the 
space  of such homomorphisms is dense $G_\delta$ in the Polish space 
of 
all homomorphisms $\tau: \FF_\infty\to [\RR ]$. 

Going back to the case of 
finite rank, it is not 
even true that a generic homomorphism $\tau:\FF_r\to[\RR ]$ 
generates the equivalence relation $\RR $.
In order to hope for a similar genericity statement, one should first answer 
the following question.

\begin{question*}
	Consider a p.m.p.\ ergodic equivalence relation $\RR $ of cost
	$<r$.
	Is it true that, in the space of homomorphisms $\tau: \FF_r\to 
	[\RR ]$ whose image generates
	$\RR $, those with dense image are dense?
\end{question*}

The fact that Bowen and then Eisenmann-Glasner had to work in the even 
smaller 
space where the first generator acts freely indicates that a Baire-categorical
approach to our main result is 
 out 
of reach at the moment, if not impossible.

\paragraph{Acknowledgements.} 
We are grateful to  Sasha Bontemps, Yves Cornulier, Gabor Elek and Todor Tsankov for their comments on preliminary versions of this work. 

The  authors acknowledge funding by the ANR project GAMME ANR-14-CE25-0004.
A.~C.\ acknowledges funding by the Deutsche Forschungsgemeinschaft (DFG, German Research Foundation) – 281869850 (RTG 2229).
D.~G. is supported by the CNRS.
F.~L.M.\ acknowledges funding by the ANR projects ANR-17-CE40-0026 AGRUME and  ANR-19-CE40-0008 AODynG.

\section{Perfect kernel for groups and minimal models}

Let $\Gamma$ be a countable discrete group. 
The topology on its space of subgroups $\Sub(\Gamma)$ admits as a basis of open 
sets the $V(\mathcal{I}, \mathcal{O})\coloneqq\{\Lambda\in \Sub(\Gamma)\colon 
\mathcal{I}\subseteq \Lambda \text{ and } \mathcal{O}\cap \Lambda=\emptyset\}$ where
$\mathcal{I}$ and $\mathcal{O}$ are finite subsets of $\Gamma$.
By the Cantor-Bendixson theorem, $\Sub(\Gamma)$ decomposes in a unique way as 
the disjoint union of a perfect set, called the \defin{perfect kernel} 
$\PK(\Gamma)$ of $\Sub(\Gamma)$, and of a countable set. 
We indicate some isolation properties of subgroups:
\begin{enumerate}
	\item \label{it: infinitely generated -> not isolated}
	If $\Lambda\in \Sub(\Gamma)$ is not finitely generated, then writing 
	$\Lambda=(\lambda_j)_{j\in \N}$ we obtain $\Lambda$ as the non-trivial 
	limit of 
	the infinite index (both in $\Lambda$ and in $\Gamma$) of the finitely 
	generated subgroups $\Lambda_n\coloneqq\la \lambda_0, \lambda_1, \cdots, 
	\lambda_n\ra$.
	
	\item \label{it: G fg f.i. subgroups are isolated}
	If $\Gamma$ is finitely generated, then its finite index subgroups are 
	isolated. Indeed, a  finite index subgroup $\Lambda$ is finitely generated 
	as well and it is alone in the open subset defined by a finite family 
	$\mathcal I$ of generators and a finite family $\mathcal O$ of 
	representatives of its cosets $ \Gamma/\Lambda$ except $\{\Lambda\}$.
	
	\item \label{it: not fg --> not isolated}
	If $\Gamma$ is not finitely generated, then its finite index subgroups are 
	also not finitely generated and thus are not isolated by Property~\ref{it: 
		infinitely generated -> not isolated}.
\end{enumerate}

Let us denote by $\Sub_{\infty i}(\Gamma)$  the subspace of infinite index 
subgroups of $\Gamma$.
The following is probably well-known but we were not able to locate a proof in 
the literature.
\begin{proposition}\label{rem: perfect kernel for free groups}
	For the free group $\FF_r$ on $r$ generators, $2\leq r\leq \infty$, 
	\begin{enumerate}[label=(\roman*)]
		\item for finite $r\geq 2$, $\PK(\FF_r)= \Sub_{\infty i}(\FF_r)$;
		\item 
		for $r$ infinite,  $\PK(\FF_\infty)=\Sub(\FF_\infty)$.
	\end{enumerate}
\end{proposition}
\begin{proof}
	We first show that if $\Lambda\in \Sub_{\infty i}(\FF_r)$, $2\leq r\leq 
	\infty$, then it is a non-trivial limit of finitely generated infinite 
	index subgroups of $\FF_r$.
	If $\Lambda$ is not finitely generated, then Property~\ref{it: infinitely 
		generated -> not isolated} above  applies.
	Thus assume $\Lambda$ is finitely generated.
	If $r$ is infinite, then $\Lambda$ has infinite index in some 
	finitely generated non-cyclic free subgroup $\Lambda\leq \Lambda*\FF_2\leq \FF_\infty$. We can thus assume that 
	the rank $r\geq 2$ is finite.
	By Hall theorem, $\Lambda$ is a free factor of a finite index subgroup 
	$\Lambda*\Delta$ of the free group $\FF_r$ (we include the case $\Lambda=\{1\}$).
	Since $\Lambda$ has infinite index, $\Delta$ is non-trivial.
		If $g\in \Delta$ is a non-trivial element, then  $\Lambda$ is the non-trivial limit of the sequence of
	finitely generated infinite index subgroups
	$(\Lambda*\langle g^n \rangle)_{n\geq 2}$ of $\FF_r$.
		
	This (with Property~\ref{it: G fg f.i. subgroups are isolated}, 
	and Property~~\ref{it: 
		not fg --> not isolated} above respectively) shows  that $\PK(\FF_r)=\Sub_{\infty 
		i}(\FF_r)$ 
	for $r<\infty$ and $\PK(\FF_\infty)=\Sub(\FF_\infty)$.
\end{proof}

\begin{remark}
This 
also shows that the Cantor-Bendixson rank of $\Sub(\FF_r)$  equals $1$ when $r$ is finite and equals $0$ when $r=\infty$.
\end{remark}

Computations of the perfect kernel for some other groups have been performed in 
\cite{BGK-2015-IRS-Lamp,skipperCantorBendixsonRank2020}.

The following is a classical result:
{\em Assume $\Gamma$ acts by homeomorphisms on a Polish space $Z$ and $\nu$ is 
	an ergodic $\Gamma$-invariant probability measure on $Z$, then the orbit of 
	$\nu$-almost every point $z\in Z$ is dense in the support of $\nu$.}
In particular:
\begin{proposition}\label{prop: dense orbit}
	If $\Gamma\acts (X,\mu)$ is a p.m.p.\ ergodic action on a standard 
	probability space, then the stabilizer $\Stab(x)$ of almost every point 
	$x\in X$ 
	has dense $\Gamma$-orbit in the support of the associated IRS $\nu=\Stab_* \mu$ of 
	$\Sub(\Gamma)$.
\end{proposition}
Thus, our main theorem
produces IRS's on $\Sub(\FF_r)$ for which almost every $\FF_r$-orbit (under 
conjugation) is dense in $\PK(\FF_r)=\Sub_{\infty i}(\FF_r)$. 
In other words, for almost every subgroup $\Lambda$ the Schreier graph of the 
action $\FF_r \acts \FF_r/\Lambda$ 
contains arbitrarily large copies of  Schreier ball of every infinite 
transitive $\FF_r$-action.
\begin{remark}
	In the introduction, we defined an IRS to be totipotent when almost every 
	subgroup has dense orbit in the perfect kernel. But an IRS can also be 
	considered as a p.m.p.\ dynamical system whose associated IRS can be 
	different. The connections between the two 
	notions of totipotency are unclear to us. Note however that since the 
	actions that we construct are totally non-free, this situation does not happen 
	$\left(\Stabmap^{\FF_r\acts \Sub(\FF_r}\right)_*\left(\Stabmap^{\alpha_\tau}_*\mu\right)=\Stabmap^{\alpha_\tau}_*\mu$
	and our IRS's are totipotent in both senses.
\end{remark}

Moreover, this proposition 
can be combined with \cite[Cor.~4.3]{glasnerUniformlyRecurrentSubgroups2015} 
to 
give the following result.

\begin{theorem}\label{thm: no minimal model}
	Let $\Gamma\acts(X,\mu)$ be a p.m.p.\ ergodic action on a standard 
	probability 
	space. Suppose that 
	the 
	support of the associated IRS contains at least two distinct minimal 
	subsets. Then the action has no minimal model.
\end{theorem} 
This is in wide contrast with free actions  of countable groups: they always admit minimal models \cite{Weiss-12-minmod}.
\begin{proof}
	By the previous proposition, the orbit closure of the stabilizer of 
	$\mu$-almost every point is equal to the support of the IRS, and hence 
	contains two distinct minimal subsets.
	Admitting a minimal model would thus be incompatible with 
	\cite[Cor.~4.3]{glasnerUniformlyRecurrentSubgroups2015}.
\end{proof}

\section{Full groups and density}

We fix once and for all a standard probability space $(X,\mu)$ and denote by 
$\Aut(X,\mu)$ the group of all its measure-preserving transformations, two such 
transformations being identified if they coincide on a full measure set. In 
order to ease notation, we will always neglect what happens on 
null 
sets. Given an element $T\in\Aut(X,\mu)$, its set of fixed points is denoted
\[
\mathrm{Fix}(T)\coloneqq\{x\in X\colon T(x)=x\}.
\]

A \defin{partial isomorphism} of $(X,\mu)$ is a partially defined Borel 
bijection 
$\varphi: \dom \varphi\to\rng\varphi$, with $\dom \varphi, \rng\varphi$ Borel 
subsets of $X$, such that $\varphi$ is measure-preserving for the measures 
induced by $\mu$ on its domain $\dom \varphi$ and its range $\rng\varphi$. In 
particular, we have $\mu(\dom\varphi)=\mu(\rng\varphi)$. The \defin{support} of 
$\varphi$ is the set 
$$\supp \varphi\coloneqq\{x\in \dom \varphi\colon \varphi(x)\neq x\}\cup\{x\in 
\rng\varphi\colon \varphi\inv(x)\neq x\}.$$

Given two partial isomorphisms with $\varphi, \psi$ disjoint domains and 
ranges, one can form their \defin{union}, which is the partial isomorphism 
\begin{align*}
\varphi\sqcup\psi:\dom \varphi\sqcup\dom \psi&\to \rng\varphi\sqcup \rng\psi\\	
x&\mapsto \left\{
\begin{array}{cl}
\varphi(x) &\text{if }x\in \dom\varphi,\\
\psi(x)&\text{if }x\in\dom\psi.
\end{array}
\right.
\end{align*}

A \defin{graphing} is a countable set of partial isomorphisms 
$\Phi$. Its \defin{cost} $\cost(\Phi)$ is the sum of the 
measures of the domains of its elements, which is also equal to the sum of 
the measures of their ranges since they preserve the measure.

Given a graphing $\Phi$, the smallest equivalence relation which contains all 
the graphs of the elements of $\Phi$ is denoted by $\RR _\Phi$ 
and called the \defin{equivalence relation generated} by $\Phi$. When 
$\Phi=\{\varphi\}$, we also write it as $\RR _\varphi$ and 
call it the equivalence relation generated by $\varphi$.

The 
equivalence relations that can be generated by graphings are called  
\defin{p.m.p.\ equivalence 
	relations}, they are Borel as subsets of $X\times X$ and have countable 
classes. The \defin{cost} $\cost(\RR )$ of a p.m.p.\ equivalence relation 
$\RR $ is the infimum of 
the costs of the graphings which generate it.

Whenever $\alpha:\Gamma\to \Aut(X,\mu)$ is a p.m.p.\ action, we denote by 
$\RR ^\alpha$ the equivalence relation generated by $\alpha(\Gamma)$. 

Given a p.m.p.\ equivalence relation $\RR $, the set of partial 
isomorphisms whose graph is contained in $\RR $ is denoted by $[[\RR]]$ and 
called the \defin{pseudo full group} of $\RR $. Here is a useful 
way of obtaining elements of the pseudo full group that we will use 
implicitly. Say that $\RR $ is 
\defin{ergodic} when every Borel $\RR $-saturated set has measure $0$ or 
$1$. Under this assumption, given any two Borel subsets $A, B\subseteq X$ of 
equal 
measure, there is $\varphi\in[[\RR ]]$ such that $\dom \varphi=A$ and 
$\rng\varphi=B$ \cite[Lem.\ 7.10]{kechrisTopicsOrbitEquivalence2004}.

The \defin{full 
	group} of $\RR $ is the subgroup  $[\RR ]$ of $\Aut(X,\mu)$ consisting of 
	almost everywhere defined 
elements 
of the pseudo full group. Endowed with 
the \defin{uniform metric} given by $d_u(S,T)=\mu(\{x\in X\colon S(x)\neq 
T(x)\})$, it becomes a Polish group. Observe that $d_u(T,\id_X)=\mu(\supp T)$.

For more material about this section, we refer to 
\cite{kechrisTopicsOrbitEquivalence2004,gaboriauOrbitEquivalenceMeasured2011} 
and the references therein.

\subsection{Around a theorem of Kittrell-Tsankov}

In this paper, we will be interested in p.m.p.\ actions 
$\tau: \FF_r\to[\RR ]$ 
with dense image in $[\RR ]$. 
To that end, the following result of Kittrell and Tsankov is very useful. Given 
a family $(\RR _i)$ of 
equivalence relations on the same set $X$, we define $\bigvee_{i\in I} \RR_i$ 
as the smallest equivalence relations which contains each $\RR _i$.

\begin{theorem}[{\cite[Thm.~4.7]{kittrellTopologicalPropertiesFull2010}}]
	\label{thm:KT}
	Let $\RR $ be a p.m.p.\ equivalence relation on $(X,\mu)$, suppose 
	that $(\RR _i)_{i\in I}$ is a family of Borel subequivalence 
	relations such that $\RR =\bigvee_{i\in I} \RR _i$. Then 
	$[\RR ]=\overline{\la \bigcup_{i\in I}[\RR _i]\ra}$.
\end{theorem}

We will also use two easy corollaries of their result which require us to set up a bit of notation.

\begin{definition}
	Given an equivalence relation $\RR$ on a set $X$ and $Y\subseteq X$, we 
	define the equivalence relation $\RR_{\restriction Y}$ \defin{restricted} 
	to $Y$ and the equivalence relation $\RR_{\induced Y}$ \defin{induced} on 
	$Y$ by
	\begin{align*}
	\RR_{\restriction Y}&\coloneqq \RR\cap Y\times Y=\{(x,y)\in \RR\colon 
	x,y\in Y\}\subseteq Y\times Y;\\
	\RR_{\induced Y}&\coloneqq \RR_{\restriction Y}\cup \{(x,x)\colon x\in 
	X\}\subseteq X\times X.
	\end{align*}
\end{definition}
Observe that given a p.m.p.\ equivalence relation $\RR $, we have a 
natural way of identifying the full group of the 
restriction $\RR_{\restriction Y}$ with the full group of the induced equivalence 
relation $\RR_{\induced Y}$ by making 
its elements act trivially outside of $Y$.

\begin{corollary}\label{cor: KT for induced}
	Let $\RR $ be an ergodic p.m.p.\ equivalence relation on $(X,\mu)$. Let
	$T\in [\RR]$ and $Y\subseteq X$ measurable such that $\mu(Y\cap T Y)>0$ and 
	put $Y_T\coloneqq \cup_{n\in\Z} T^nY$. Then 
	$\overline{\la T,[\RR_{\induced Y}]\ra}\geq [\RR_{\induced 
	Y_T}]$.
\end{corollary}
\begin{proof}	
    Since $\mu(Y\cap TY)>0$ and $\RR$ is ergodic, we have that $\RR_{\induced 
    Y\cup TY}=\RR_{\induced Y}\vee \RR_{\induced TY}$. Therefore Theorem 
    \ref{thm:KT} implies that \[\overline{\la[\RR_{\induced 
    Y}],T[\RR_{\induced Y}]T^{-1}\ra}=[\RR_{Y\cup TY}].\]  
	Now observe that $(Y\cup TY)\cap T(Y\cup TY)\supseteq TY$ has positive 
	measure. Therefore Theorem \ref{thm:KT} implies that  $\overline{\la 
	T,[\RR_{\induced Y}]\ra}$ contains $[\RR_{\induced (Y\cup TY\cup T^2Y)}]$ 
	and the corollary follows by induction. 
 \end{proof}

\begin{corollary}\label{crl: erg+Amu1/2}
  Consider an ergodic p.m.p.\ equivalence relation $\RR$ on $(X,\mu)$ and let 
  $Y\subseteq X$ be a positive measure subset. Let $\alpha$ be a p.m.p.\ action 
  of $\Gamma$ on $(X,\mu)$ such that $\alpha(\Gamma)\leq [\RR]$, 
  $\mu(\alpha(\Gamma)Y)=1$ and $[\RR_{\induced Y}]\leq \overline 
  {\alpha(\Gamma)} $. Then either $\overline{\alpha(\Gamma)}=[\RR]$, or 
$\Gamma$ preserves a finite partition $\{Y_i\}_{i=1}^k$ of $X$, with $Y\subseteq Y_1$ and 
$[\RR_{\induced Y_i}]\leq \overline {\alpha(\Gamma)} $
for each $i\leq k$.

In particular if $\mu(Y)>1/2$, then $k=1$ and hence 
$\overline{\alpha(\Gamma)}=[\RR]$.
\end{corollary}
\begin{proof}
  Let $B\supset Y$ be a subset of maximal measure such that $\overline {\alpha(\Gamma)} \geq [\RR_{\induced B}]$. Then by the above corollary for every 
  $\gamma\in \Gamma$  such that $\alpha(\gamma) B\neq B$, we must have that 
  $\mu(B\cap 
  \alpha(\gamma) B)=0$, hence $B$ is an atom of a finite partition preserved by 
  the $\Gamma$-action $\alpha$. 
\end{proof}

\subsection{From graphings to density}

The following is a slight variation of 
\cite[Def.~8]{lemaitrenumbertopologicalgenerators2014}.

\begin{definition}\label{def:pre-cycle}
	Let $n\geq 2$. A \defin{pre-cycle} of \defin{length} $n$ is a partial 
	isomorphism $\varphi$ 
	such that if we set $B\coloneqq \dom\varphi\setminus\rng\varphi$ (the \defin{basis} of the pre-cycle), then 
	$\{\varphi^i(B)\}_{i=0,\ldots,n-2}$ is a partition of $\dom\varphi$ and 
	$\{\varphi^i(B)\}_{i=1,\ldots,n-1}$ is a partition of $\rng\varphi$.
	
	We say that $T\in \Aut(X,\mu)$ \defin{extends} $\varphi$ if $T x=\varphi x$ 
	for every $x\in\dom(\varphi)$.
\end{definition}

Observe that a pre-cycle of length $2$ is an element $\varphi\in [[\RR]]$ such 
that $\dom(\varphi)\cap \rng(\varphi)=\emptyset$. If $\varphi$ is a 
pre-cycle of length $n$, then $\mu(\supp \varphi)=n\mu(B)$ and 
$\mu(\dom\varphi)=(n-1)\mu(B)$.

A \defin{$n$-cycle} is a measure-preserving transformation all of whose orbits have cardinality either $1$ or $n$.
Given a pre-cycle $\varphi$ of length $n$, 
we can extend it to an $n$-cycle 
$U_\varphi\in[\RR _\varphi]$ as follows:
\[U_\varphi(x)\coloneqq\left\{
\begin{array}{cl}
\varphi(x)&\text{if }x\in\dom\varphi;\\
\varphi^{-(n-1)}(x)& \text{if }x\in\rng\varphi\setminus\dom\varphi;\\
x& \text{otherwise.}
\end{array}
\right.\]
This $n$-cycle $U_\varphi$ is called the \defin{closing cycle} of 
 $\varphi$ and $\supp U_\varphi=\supp \varphi$.

\begin{remark}\label{rmk: defequiv}
Note that if $\{\varphi_1,...,\varphi_{n-1}\}$ is a pre-$n$-cycle in the sense 
of \cite[Def.~8]{lemaitrenumbertopologicalgenerators2014}, then 
$\varphi_1\sqcup\cdots\sqcup \varphi_{n-1}$ is a pre-cycle of length $n$ in our 
sense, 
and that if $\varphi$ is a pre-cycle of length $n$ in our sense then 
$\{\varphi_{\restriction \varphi^i(B)}\colon i=0,...,n-2\}$ is a pre-$n$-cycle 
in the sense of \cite[Def.~8]{lemaitrenumbertopologicalgenerators2014}. The 
reason for this change of terminology will become apparent in the statement of 
the next lemma, which was proved for $U=U_\varphi$ in
\cite[Prop.~10]{lemaitrenumbertopologicalgenerators2014}.
\end{remark}

\begin{lemma}\label{lem: conj trick using pre-cycle}
	Suppose $\varphi$ is a pre-cycle of basis $B$,	let 
	$\psi\coloneqq\varphi_{\restriction B}$, and suppose $U\in \Aut(X,\mu)$ 
	extends 
	$\varphi$. Then $[\RR _\varphi]$ is  
	contained in the closure of the group generated by 
	$[\RR _\psi]\cup\{U\}$.
\end{lemma}
\begin{proof}
	Let $n$ be the length of $\varphi$. For $i=0,...,n-2$ let 
	$\psi_i=\varphi_{\restriction \varphi^i(B)}$, then we 
	have $\RR _\varphi=\bigvee_{i=0}^{n-2}\RR _{\psi_i}$. Since $U$ 
	extends $\varphi$, we have $U\psi_iU\inv=\psi_{i+1}$ for all $i=0,...,n-3$, 
	and hence $U[\RR _{\psi_i}]U\inv=[\RR _{\psi_{i+1}}]$.
	Since $\psi_0=\psi$, the group generated by $U\cup 
	[\RR _\psi]$ contains $[\RR _{\psi_i}]$ for all $i=0,...,n-2$. 
	Theorem \ref{thm:KT} finishes the proof.
\end{proof}

	The following proposition is obtained by a slight modification of the 
	proof of the main theorem of 
	\cite{lemaitrenumbertopologicalgenerators2014}.

\begin{proposition}\label{prop: the good pre-cycles}
	Let $\RR$ be a p.m.p.\ ergodic equivalence relation on $X$ and let 
	$Y\subseteq X$ 
	be a positive measure subset. Let $\RR_0\leq \RR_{\induced Y}$ be a 
	hyperfinite equivalence relation whose restriction to $Y$ is ergodic (and trivial on $X\setminus Y$). 
	Suppose that 
	$\cost(\RR_{\induced Y})<r\mu(Y)$ for some integer $r\geq 2$. 
		Then there are $r-1$ pre-cycles $\varphi_2,\varphi_3,\ldots,\varphi_r\in [[\RR_{\induced Y}]]$  
		such that 
	$\mu(\supp(\varphi_i))<\mu(Y)$ and such that whenever $U_2,U_3,\ldots,U_r\in 
	[\RR]$ 
	extend $\varphi_2,\varphi_3,\ldots,\varphi_r$, we have  
	$\overline{\la[\RR_0],U_2,U_3,\ldots,U_r\ra}\geq [\RR_{\induced Y}]$.
\end{proposition}
For instance, one can take $U_2,U_3,\ldots,U_r$ to be the closing cycles of 
$\varphi_2,\varphi_3,\ldots,\varphi_r$.

\begin{proof}

	Let $T\in [\RR_0]$ be such that its restriction to $Y$ is ergodic. 
	Our assumption $\cost(\RR_{\induced Y})<r\mu(Y)$ means that the normalized cost of
	the restriction $\RR_{\restriction Y}$ is less than $r$.
	Lemma 
	III.5 from 
	\cite{Gab00} then provides a graphing
	$\Phi$  on $Y$ of normalized cost $<(r-1)$ such that $\{T_{\restriction Y}\}\cup\Phi$ generates 
	the restriction $\RR_{\restriction Y}$. We now view $\Phi$ as a graphing on $X$, so that 
	$\{T\}\cup\Phi$
	generates $\RR_{\induced Y}$, and $\cost(\Phi)<(r-1)\mu(Y)$.
	Let $c\coloneqq \cost(\Phi)/(r-1)<\mu(Y)$. We take $p\in\N$ so large that 
	$c(p+2)/p<\mu(Y)$.
	
	Pick $\psi\in[[\RR_0]]$ a pre-cycle of length $2$ whose domain $B$ has 
	measure $c/p$. By cutting and pasting the elements of 
	$\Phi$ and by conjugating them by elements of $[\RR_0]$, we may as well 
	assume that $\Phi=\{\varphi_2,...,\varphi_r\}$ 
	where each $\varphi_i$ is a pre-cycle of length $p+2$ extending $\psi$ of 
	basis $B$ whose 
	support is a strict subset of $Y$. 
	Assume that $U_i\in [\RR]$ extends $\varphi_i$ for every $i=2,3,\ldots,r$. 
	Since $\psi\in [[\RR_0]]$, then $[\RR_\psi]\leq [\RR_0]$. We can apply 
	Lemma \ref{lem: conj trick using 
		pre-cycle} and obtain that the closure of the group generated by 
		$[\RR_0]$ and $U_i$ contains $[\RR_{\varphi_i}]$.
	Since  
	$\RR_{\induced Y}=\RR _{0}\vee\RR 
	_{\varphi_2}\vee\cdots\vee\RR_{\varphi_r}$, we can
	conclude the proof of the theorem using Theorem \ref{thm:KT}.
\end{proof}

\begin{remark}
Observe that we have a lot of freedom in the construction of the pre-cycles $\varphi_2,\varphi_3,\ldots,\varphi_r$ of Proposition~\ref{prop: the good pre-cycles}.
To start with, their length can be chosen to be any integer $n=p+2$ large enough that $\frac{c}{\mu(Y)}< \frac{n-2}{n}$.
Actually, they could even have been chosen with any (possibly different) lengths $n_2, n_3,\cdots, n_r$, large enough integers so that 
$\frac{c}{\mu(Y)}< \frac{n_j-2}{n_j}$:
simply pick $r-1$ pre-cycles $\psi_j\in[[\RR_0]]$ of length $2$ whose domain $B_j$ has 
	measure $\frac{c}{n_j-2}$ and proceed as in the proof above.
	
	In particular, the periodic closing cycles $U_2,U_3, \cdots, U_r$ can be assumed to have any large enough period $n_2, n_3, \cdots, n_r$ and domains contained in $Y$ of measure $<\mu(Y)$.
	Up to conjugating by elements of $[\RR_0]$, one can further assume the closing cycles have in $Y$ a non-null common subset of fixed points: $\mu(\mathrm{Fix}(U_2)\cap \mathrm{Fix}(U_3)\cap\cdots\cap\mathrm{Fix}(U_r)\cap Y)>0$.
\end{remark}

\section{Evanescent pairs and topological generators}

In this section our main goal is to obtain two topological generators of the 
full group of a 
hyperfinite 
ergodic equivalence 
relation with new flexibility properties relying on the following definition.

\begin{definition}
\label{def: evanescent pair}
A pair $(T,V)$ of elements of the full group $[\RR]$ of the 
p.m.p.\ equivalence relation $\RR$ is called an \defin{evanescent pair of 
topological generators} of $\RR$ if 
\begin{enumerate}
\item $V$ is periodic; and
\item for 
	every 
	$n\in\N$, the full group $[\RR]$ is topologically generated by the conjugates of $V^n$ by the
	powers of $T$, i.e., $\overline{\langle T^j V^n T^{-j}:{j\in \Z}\rangle}=[\RR]$.
\end{enumerate}
\end{definition}
In particular, if $(T,V)$ is an evanescent pair of topological generators, then 
the following hold: 
\begin{itemize}
\item  the pair $(T,V)$ topologically generates $[\RR]$,
\item $(T,V^n)$ is an evanescent pair of topological generators  for any $n\in 
\N$,
\item $d_u(V^{n!},\id_X)$ tends to $0$ when $n$ tends to $\infty$.
\end{itemize}

We will show in Theorem~\ref{Th: existence of evanescent pair} that the 
odometer $T_0$ 
can be completed to form an evanescent pair $(T_0,V)$ of topological generators for 
$\RR_{T_0}$, and that  the set of possible $V$ is actually a 
dense 
$G_\delta$. 

In this section, we set $X=\{0,1\}^\N$ endowed with the Bernoulli $1/2$ 
measure 
$\mu=(\frac12\delta_0+\frac12\delta_1)^{\otimes\N}$. 
Given $s\in \{0,1\}^n$, we define the basic clopen set $$N_s\coloneqq\{x\in 
\{0,1\}^\N\colon
x_i=s_i\text{ for } 1\leq i\leq n\}.$$
The \defin{odometer} $T_0$ 
is the measure-preserving transformation of this space defined as adding the binary 
sequence $(1,0,0,\cdots)$ with carry to the right. More precisely, for each 
sequences $x\in\{0,1\}^\N$, if $k$ is the (possibly infinite) 
first integer such that $x_k=0$, then  $y=T_0(x)$ is defined by
\[
y_n\coloneqq\left\{
\begin{array}{cl}
0&\text{if }n< k,\\
1&\text{if }n= k,\\
x_n&\text{if }n>k.
\end{array}
\right.
\]

For each $n\in\N$, the permutation group $\Sym(\{0,1\}^n)$ has a natural action $\alpha_n$ on
$\{0,1\}^\N\simeq\{0,1\}^n\times\{0,1\}^\N$ given 
for $x\in\{0,1\}^\N$ and $\sigma\in\Sym(\{0,1\}^n)$ by:
$$
\alpha_n(\sigma) (x_1,...,x_n, 
x_{n+1},...)\coloneqq(\sigma(x_1,...,x_n),x_{n+1},...).
$$

The sequence $(\alpha_n(\Sym(\{0,1\}^n)))_{n\in\N}$ is an increasing sequence 
of subgroups of the full group $[\RR_{T_0}]$ whose reunion is dense in 
$[\RR_{T_0}]$ (see 
\cite[Prop.~3.8]{kechrisGlobalaspectsergodic2010}).

We now define a sequence of involutions 
$U_n\in[\RR_{T_0}]$ 
with disjoint supports as in 
\cite[Sect.~4.2]{lemaitrefullgroupsnonergodic2016}: 
$U_n\coloneqq\alpha_n(\upsilon_n)$ 
where $\upsilon_n\in\Sym(\{0,1\}^n)$ is the $2$-points support transposition 
that exchanges $0^{n-1}1$ 
and $1^{n-1}0$.
Observe that $U_n$ is
the involution with support $N_{1^{n-1}0}\sqcup N_{0^{n-1}1}$ (of measure $2^{-n+1}$)  which is equal to 
$T_0$ on $N_{1^{n-1}0}$ and $T_0\inv$ on $N_{0^{n-1}1}$.

Recall that if $\tau_n\in\Sym(\{0,1\}^n)$ is $2^n$-cycle and $w_n$ is a 
transposition which exchanges two $\tau_n$-consecutive elements, then 
the group $\Sym(\{0,1\}^n)$ is generated by the conjugates of $w_n$ by 
powers of $\tau_n$ (actually $2^n-1$ of them are enough).  
A straightforward modification gives the following (see \cite[Lem.\ 
4.3]{lemaitrefullgroupsnonergodic2016} for a detailed proof):

\begin{lemma}\label{lem: conjugacy generation for T0}
	For every $n\in\N$, the group $\alpha_n(\Sym(\{0,1\}^n))$ is contained in 
	the group generated by the conjugates of $U_n$ by powers of $T_0$. 
\end{lemma}

Given a periodic p.m.p.\ transformation $U$ 
and $k\in\N$, we say that $V$ is a \defin{$k$th root} of $U$ when $\supp 
U=\supp V$ and 
$V^{k}=U$. 
The following lemma is well-known.  

\begin{lemma}
	Whenever $\RR $ is an ergodic equivalence relation, 
	every periodic element in $[\RR ]$ admits a $k$th root in 
	$[\RR]$.
\end{lemma}
\begin{proof}
	Let us first prove that every $n$-cycle $U\in [\RR ]$ admits a $k$th 
	root. To this end, 
	pick a fundamental domain $A$ for the restriction of $U$ to its support. 
	Since $\RR$ is ergodic, we can pick a $k$-cycle $V\in[\RR ]$ supported on 
	$A$. Let $B$ be a 
	fundamental domain for $V$, and put $C\coloneqq A\setminus B$. Then it is 
	straightforward to 
	check that  $W\in[\RR ]$ defined as follows is a $k$th root of $U$: 
	$$
	W(x)\coloneqq\left\{
	\begin{array}{cl}
	UU^iVU^{-i}(x) &\text{if } x\in U^i(B), \\
	U^iVU^{-i}(x)& \text{if } x\in U^i(C),\\
	x&\text{otherwise.}
	\end{array}
	\right.
	$$
	In the general case, one glues together the $k$th roots obtained for every 
	$n\in\N$ by considering the 
	restrictions of $U$ to $U$-orbits of cardinality $n$.
\end{proof}
\begin{remark}
	The same proof works more generally for \emph{aperiodic} p.m.p.\ 
	equivalence relations.
\end{remark}

\begin{theorem}\label{Th: existence of evanescent pair}
	The set of $V\in[\RR_{T_0}]$  such that $(T_0,V)$ is an evanescent pair of 
	topological generators of $\RR_{T_0}$ is 
	a dense $G_\delta$ subset of $[\RR_{T_0}]$. 
\end{theorem}

\begin{proof}
Denote by $\mathcal P$ the set of periodic elements of $[\RR_{T_0}]$. It is a 
direct consequence of Rokhlin's lemma that $\mathcal P$ is dense 
in $[\RR_{T_0}]$.
And similarly the subset $\mathcal P'\subseteq \mathcal P$ of $V\in[\RR_{T_0}]$ 
with finite order 
(or equivalently, with bounded orbit size) is dense in $[\RR_{T_0}]$.

Writing $\mathcal P$ as the intersection (over the positive integers $q$) of 
the open 
sets $\{V\in [\RR_{T_0}]\colon \exists p\in \N,\ d(V^{p!},\id_X)<1/q\}$ shows 
that 
$\mathcal P$ is a $G_\delta$ subset of $[\RR_{T_0}]$.

Denote by $\mathcal E$ the set of $V\in[\RR_{T_0}]$ such that for every $n$, 
the 
group $[\RR_{T_0}]$ is topologically generated by conjugates of $V^n$ by 
powers of $T_0$. We want to show that $\mathcal P\cap \mathcal E$ is dense 
$G_\delta$, and since $\mathcal P$ is dense $G_\delta$ it suffices (by the 
Baire category theorem in the Polish group $[\RR_{T_0}]$) to show that 
$\mathcal E$ is dense $G_\delta$. 

 For every $m,n\in\N$, set 
 \[
 \mathcal E_{m,n}\coloneqq\left\{V\in [\RR_{T_0}]\colon
 \alpha_n(\Sym(\{0,1\}^n))\leq\overline{\la T_0^kV^mT_0^{-k}\colon k\in\Z \ra}
 \right\}.
 \] 
 The density of the union of the  
 $\alpha_n(\Sym(\{0,1\}^n))$ in $[\RR_{T_0}]$ recalled above implies that 
 $\mathcal 
E=\bigcap_{m,n\in\N}\mathcal E_{m,n}$. So it suffices to show that each 
$\mathcal 
E_{m,n}$ is dense $G_\delta$.

Let us first check that each $\mathcal E_{m,n}$ is $G_\delta$. Denote by 
$\mathbf W$ 
the subgroup of $\FF_2=\la a_1,a_2\ra$ generated by the conjugates of $a_2$ 
by powers of $a_1$. So for $w=w(a_1,a_2)\in \mathbf W$ and $V\in[\RR_{T_0}]$, 
the element $w(T_0, V^m)$  is a product of conjugates of 
$V^m$ by powers of $T_0$.
By the definition of the closure we can write $\mathcal 
E_{m,n}$ as 
$$\mathcal E_{m,n}=\bigcap_{p\in\N}\ \bigcap_{\sigma\in 
\Sym(\{0,1\}^n)}\ \bigcup_{w\in \mathbf W}\left\{V\in [\RR_{T_0}]\colon 
d_u(w(T_0,V^m),\sigma)<\frac 1p\right\}.
$$

Since the 
map $V\mapsto w(T_0,V)$ is continuous, each of the above right-hand sets is 
open, so 
their reunion over $w\in \mathbf W$ is also open, and we conclude that 
$\mathcal 
E_{m,n}$ is 
$G_\delta$.

To check the density, it suffices to show that, for each $m,n$, one can approximate 
arbitrary elements of $\mathcal P'$ by elements of $\mathcal E_{m,n}$. So let 
$U\in 
\mathcal P'$ and let $\epsilon>0$. Denote by $K$ the order of $U$. Pick $p\geq 
n$ such that $2^{-p}K<\epsilon/2$. Let $A$ be the 
$U$-saturation of the support of $U_p=\alpha_p(\upsilon_p)$ (defined at the beginning of the section). The measure of $A$ is at most $\epsilon$. Finally, 
let $V$ be a $(Km)$th root of $U_p$ and define 
$$
\tilde U(x)\coloneqq\left\{
\begin{array}{cl}
U(x) & \text{if }x\in X\setminus A,\\
V(x) & \text{if } x\in A.
\end{array}
\right.
$$
 By 
construction  $d_u(U,\tilde U)\leq\mu(A)<\epsilon$.
Observe that $\tilde U^{Km}=(\tilde U^m)^K=U_p$, thus Lemma \ref{lem: 
conjugacy generation for T0} yields that $\tilde U \in \mathcal E_{Km,p}\subseteq \mathcal E_{m,p}$.
Since $p\geq n$, $\tilde U \in \mathcal E_{m,p}\subseteq
\mathcal E_{m,n}$,
 so we are done.
\end{proof}

Let us make a few comments on the above result. 
First, one can check that the pair $(T_0,V)$  produced in the construction of 
\cite[Thm.\ 4.2]{lemaitrefullgroupsnonergodic2016} provides an explicit example 
of an 
evanescent pair of topological generators of $\RR_{T_0}$. Also,
the above proof can be adapted to show that any \emph{rank one}  p.m.p.\ 
ergodic 
transformation \cite[Sect.\ 8]{ORW-1982} can be completed to form 
an evanescent pair of topological generators (see \cite[Thm. 
5.28]{lemaitreGroupesPleinsPreservant2014} 
for an 
explicit example of a pair which is evanescent). Proving these results is beyond the scope 
of this 
paper, so we leave it as an exercise for the 
interested reader.

 It is unclear whether every 
	p.m.p.\ ergodic transformations can be completed to form an evanescent pair 
	of 
	topological generators for its full group. 	
	Nevertheless, we can characterize the existence of an evanescent pair as 
	follows.
	
	\begin{theorem}\label{thm: chara evanescent}
		Let $\RR $ be an ergodic p.m.p.\ equivalence relation, then 
		$\RR$ admits an evanescent pair of topological generators if 
		and only if $\RR $ has cost~$1$. 
	\end{theorem}
	\begin{proof}
		If $\RR $ admits an evanescent pair $(T, V)$, then since $V$ is 
		periodic we have $\mu(\supp V^{n!})\to 0$. Since any set of topological 
		generators 
		for 
		$[\RR ]$ generates the equivalence relation $\RR $, we 
		conclude that $\RR $ has cost $1$. 
		
		As for the converse, 
		Theorem 4 and 5 from \cite{dyeGroupsMeasurePreserving1959} provide an 
		ergodic hyperfinite 
		subequivalence relation which is isomorphic to that of the
		odometer. So we can pick a conjugate of the odometer 
		$T\in[\RR]$. Repeating the proof of Theorem \ref{Th: existence of 
		evanescent 
		pair}, 
		we see that the set $\mathcal E_{T}$ of $V\in [\RR ]$ such that 
		for 
		every $n\in\N$, $[\RR_{T}]$ is \emph{contained} in $\overline{\langle 
		T^j 
		V^n 
			T^{-j}:{j\in \Z}\rangle}$ is dense $G_\delta$ in $[\RR ]$.

		Let us now consider the set $\mathcal E_{\RR}$ of $V\in[\RR ]$ such 
		that  
		$(T,V)$ is an evanescent pair of topological generators of $\RR$, 
		and for $n\in\N$ the set $\mathcal E_n$ of $V\in [\RR ]$ such 
		that 
		$V$ is periodic and $\overline{\langle T^j V^n T^{-j}:{j\in 
				\Z}\rangle}=[\RR]$. Each $\mathcal E_n$ is $G_\delta$ by the same 
				argument as 
		in the proof of Theorem \ref{Th: existence of evanescent pair}. Since 
		$\mathcal 
		E_{\RR}=\bigcap_n \mathcal E_n$, it suffices to show that each $\mathcal E_n$ 
		is 
		dense in order to apply the Baire category theorem and finish the 
		proof. 
		
		Let us fix $n\in\N$. Since $\mathcal E_{T}$ is dense in $[\RR]$, we 
		only need to 
		approximate elements of $\mathcal E_{T}$ by elements of $\mathcal 
		E_n$. Moreover, the set of $V\in [\RR ]$ such that $\mu(\supp 
		V)<1$ 
		is open and dense, so we only need to approximate every $V\in \mathcal 
		E_T$ with $\mu(\supp V)<1$ by elements of $\mathcal E_n$.

		So let $V\in \mathcal E_T$ with $\mu(\supp V)<1$, and take 
		$\epsilon>0$. 
		
		Since $\RR $ has cost $1$, Lemma III.5 from \cite{Gab00} yields a 
		graphing $\Phi$ of cost 
		$<\frac 13\min(\epsilon,\mu(X\setminus \supp V))$ such that 
		$\{T\}\cup\Phi$ 
		generates $\RR $. Conjugating by elements of $[\RR_{T}]$ and pasting 
		the elements of 
		$\Phi$, we may as well assume that $\Phi=\{\varphi\}$ where 
		$\mu(\dom\varphi)<\epsilon/3$ and $\varphi$ is a pre-cycle of length 
		$2$ whose 
		support disjoint from 
		$\supp 
		V$. We then pick $\psi\in[[\RR_{T}]]$ such that $\varphi\sqcup \psi$ 
		is a 
		pre-cycle of length $3$ of support disjoint from $\supp V$, and denote 
		by $U_1$ the 
		associated $3$-cycle. 
		
		Now let $U_2$ be an $n$th root of $U_1$, let $V_2\coloneqq VU_2$, then 
		$d_u(V_2,V)<\epsilon$ and we claim that $V_2$ belongs 
		to $\mathcal E_n$. In order to prove this, let us denote by $G$ the 
		closed 
		group generated by the conjugates of $V_2^n$ by powers of $T$.
		
		Since $U_2$ and $V$ have disjoint support, they commute, and so 
		$V_2^n=U_2^nV^n=U_1V^n$. So $(V_2^n)^3=V^{3n}$, and since 
		$V\in \mathcal E_T$, we have that $[\RR_{T}]\leq G$. In particular 
		$[\RR _{\psi}]\leq G$,
		 and conjugating by $V_2^n$ (which acts as 
		$U_1$ on $\supp U_1$), we get that $[\RR _\varphi]\leq G$ (see also 
		Lemma \ref{lem: conj trick using pre-cycle}). 
		Since $\RR =\RR _{T}\vee \RR _{\psi}$, we conclude 
		by 
		Theorem \ref{thm:KT} that 
		$G$ contains $[\RR ]$ as wanted.
\end{proof}

\section{Proof of the main theorem}
\label{sect: Proof of the main theorem}

As shown in Proposition \ref{rem: perfect kernel for free 
groups}, the perfect kernel of $\Sub(\FF_r)$, $1<r<\infty$, is the space of infinite 
index 
subgroups. We will construct a p.m.p.\ action of $\FF_r$ 
for which almost every Schreier graph contains all possible balls of Schreier graphs 
of transitive $\FF_r$-actions on infinite sets. 

\begin{step}[using a smaller subset]\label{step: smaller set}
We start with a p.m.p.\ ergodic equivalence relation $\RR $ on $(X,\mu)$ 
of cost $<r$. By the induction formula \cite[Proposition II.6]{Gab00}, there is 
a subset $Y\subseteq X$ such that $1/2<\mu(Y)<1$ and 
such that the (normalized) cost of the restriction $\RR _{\restriction 
Y}$ is still $<r$.  
Thus the cost of the induced equivalence relation $\RR _{\induced Y}$ is $<r\mu(Y)$.

Using results of Dye \cite[Thm.\ 4 and 5]{dyeGroupsMeasurePreserving1959} as in 
the proof of Theorem \ref{thm: chara evanescent}, 
one can pick a conjugate of the odometer 
	$T\in [\RR _{\restriction Y}]$. We view $T$ as an element of $[\RR 
	_{\induced Y}]$.

Now we apply Proposition \ref{prop: the good pre-cycles} (where $\RR_0=\RR_T$) 
to obtain pre-cycles 
$\varphi_2,\ldots,\varphi_r\in [[\RR_{\restriction Y}]]$ whose supports have 
measure $<\mu(Y)$. For $i\leq r$, we let 
$U_i$ be the closing cycle of $\varphi_i$ as defined after 
Definition \ref{def:pre-cycle}. 
Set $\eta\coloneqq \mu(Y\setminus \supp U_2)>0$. 
Let $m_0$ be a positive integer such that $\mu(X\setminus Y)/m_0< \eta/2$.
\end{step}

\begin{step}[preparing the finite actions]
Let $(G_n)_{n\geq 1}$ be an enumeration of the (finite radius) balls of the 
Schreier graphs of all the
transitive $\FF_{r}$-actions over an infinite set, up to labeled graph 
isomorphism, and for which the number of vertices satisfies $\abs{G_n}\geq m_0$.

Since $G_n$ comes from a transitive action over an infinite set, we can choose 
some 
$\ell\in\{1,...,r\}$ and some $\zeta_n \in G_n$ such that 
there is no $a_\ell$-labeled edge whose source is equal to $\zeta_n $.

Pick $\delta_n ,\xi_n\not\in G_n$, set $G'_n\coloneqq G_n\sqcup\{\delta_n ,\xi_n\}$ and add an $a_\ell$-edge from $\zeta_n $ to $\delta_n $, an $a_1$-edge from $\delta_n $ to $\xi_n$ and an $a_2$-edge from $\xi_n$ to itself. 

In this way we obtain a finite \textit{partial Schreier graph} and this can be 
extended to a 
genuine Schreier graph of an $\FF_r=\la a_1,...,a_r\ra$-action on the same set as 
follows: for 
each $i\in 
\{1, 2, \cdots, r\}$, we consider the connected components of the subgraph 
obtained by 
keeping only the edges labeled $a_i$. These are 
either cycles (we don't modify them) or oriented segments 
(possibly reduced to a single vertex), in which case we add one edge labelled 
$a_i$ from the end  to the 
beginning of 
the segment.

Therefore we obtain an \textit{action $\rho_n$ of $\FF_r$ on the finite set 
$G'_n$ and a special point $\xi_n\in 
G'_n\setminus G_n$ such that $\rho_n(a_2)\xi_n=\xi_n$.} 
\end{step}

\begin{step}[defining the action]
Set $C\coloneqq X\setminus Y$. Consider a partition $C=\sqcup_{n\geq 1}C_n$ 
where $\mu(C_n)>0$ for every $n$. 
We are going to define an amplified version of the action 
$\rho_n$ on $C_n$ as follows. 

For each $n\geq 1$, we take a measurable partition $C_n=\sqcup_{g\in G'_n} 
B_n^g$ such that $\mu(B_n^g)\abs{G'_n}=\mu(C_n)$ for every $g\in G'_n$. Set 
$B_n\coloneqq B_n^{\xi_n}$. Using ergodicity of $\RR$, for every $g\in 
G'_n\setminus\{\xi_n\}$ we choose $\psi_g: B_n\to B_n^g$ in the pseudo full 
group $[[\RR]]$ of $\RR$. In this way we obtain an action $\alpha_n$ of $\FF_r$ 
defined on $C_n$ by the formula \[\text{ if }x\in B_n^{g_0} \text{ and }
 \rho_n(\gamma)g_0=g_1
 \text{ then } \alpha_n(\gamma)x\coloneqq \psi_{g_1}\psi_{g_0}^{-1}(x),\]
 and trivial on $X\setminus C_n$. Thus  $\alpha_n(\FF_r)\leq  [\RR_{\induced C_n}]$.

Gluing all the $\alpha_n$ together, we obtain an action $\alpha_\infty$ of 
$\FF_r$ on $X$ such that $\alpha_\infty(\FF_r)\leq  [\RR_{\induced C}]$ and 
such that $\alpha_\infty$ restricted to $C_n$ is $\alpha_n$.

Let $T\in [\RR_{\induced Y}]$ be the conjugate of the odometer introduced in Step 
\ref{step: smaller set}.
Theorem \ref{Th: existence of evanescent pair} states that the set of $V\in [\RR_T]$  
such that $(T,V)$ is an evanescent pair of generators
for $\RR_T$ is dense so we can choose such a $V$ with $\mu(\supp V)<\eta/2$.
 Let $W\in [\RR_T]$ be such that 
 $\mu(\supp(WU_2W^{-1})\cap \supp V)=0$. Set
\begin{itemize}
	\item $B\coloneqq \cup_n B_n$ and remark that $\mu(B)\leq \mu(C)/m_0<\eta/2$; 
	\item $D\coloneqq  \supp(WU_2W^{-1})\cup \supp V$ and observe that $\mu(Y\setminus D)>\eta/2$.
\end{itemize}

 Therefore there exists a subset $A\subseteq Y\setminus D$ of measure $\mu(A)=\mu(B)$. Let $I\in [\RR]$ be an involution with support $A\cup B$ and which exchanges $A$ and $B$. 

We finally define the desired action $\alpha$ of $\FF_r$ by setting 
\begin{align*}
\alpha(a_1)\coloneqq\ &T\alpha_\infty(a_1),\\ \alpha(a_2)\coloneqq\ & V(WU_2W^{-1}) (I \alpha_\infty(a_2)),\\ \alpha(a_i)\coloneqq\ &U_i\alpha_\infty(a_i)\qquad\text{ for }i\geq 3.
\end{align*}
See Figure~\ref{fig: desired action} for the action of $\alpha(a_2)$.
Note that $a_2$ is the only generator of $\FF_r$ which does not leave the set $Y$ invariant because of the
presence of the involution $I$ in the definition of its action.

\begin{figure}[h]
	\includegraphics{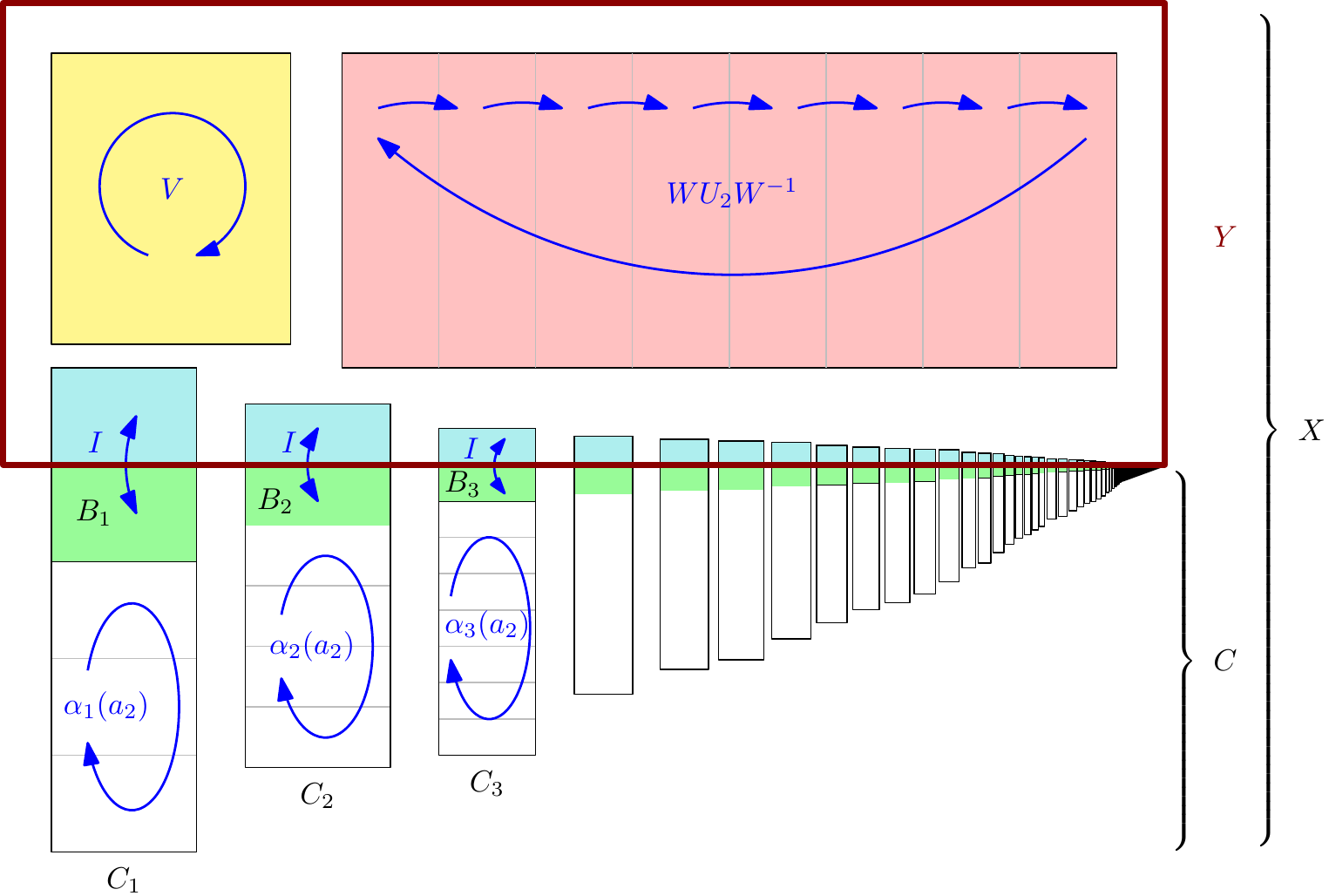}
	\caption{The action of $\alpha(a_2)$ on $X$. }
	\label{fig: desired action}
\end{figure}
\end{step}

\begin{step}[density]\ 

\begin{enumerate}[label=(\alph*)]
\item We claim that $\overline{\alpha(\FF_r)}\geq  [\RR_T]$.
\end{enumerate}

Indeed let $S\in [\RR_T]$ and let us fix $\epsilon>0$. There exists $n_0$, such 
that if we set \(C_{> n_0}\coloneqq \cup_{n> n_0}C_n\) then \(\mu(C_{> 
n_0})<\epsilon/2\). The elements $U_2$, $I$, $\alpha_1(a_2), \ldots, 
\alpha_{n_0}(a_2)$ have uniformly bounded orbits. So we can pick $k\in\N$ such 
that 
$U_2^k$, $I^k$ and $\alpha_1(a_2)^k, \ldots, 
\alpha_{n_0}(a_2)^k$ are the identity.

 By construction  $V$, $WU_2W^{-1}$, $I$ and $\alpha_\infty(a_2)$ have mutually 
 disjoint supports and hence commute. Therefore 
 $\alpha(a_2)^k=V^k\alpha_\infty(a_2)^k$.
 
 The crucial assumption that $(T,V)$ is an evanescent pair of generators now 
 comes into play: there is a word 
 $w(T,V^k)$ which is a product of conjugates of $V^k$ by powers of $T$ such 
 that $d_u(w(T,V^k),S)<\epsilon/2$. Remark that 
 $\alpha(a_1)$ acts on $Y$ the same way as $T$, and 
 that   $\alpha(a_2)^k$ acts on $Y$ the same way as $V^k$.
 Also note that $\alpha(a_1)$ preserves each $C_j$ while $\alpha(a_2)^k$ is the identity on each $C_j$, $j=1, 2, \cdots, n_0$, so that for all  $m\in \Z$, the transformation
 $\alpha(a_1)^m\alpha(a_2)^k\alpha(a_1)^{-m}$ acts on $C_1\cup C_2\cup \cdots\cup C_{n_0}$ as the identity.

 It now follows from the fact that $w$ is a product of conjugates of $V^k$ by powers of $T$ that  $w(T,V^k)$ 
 and $w(\alpha(a_1),\alpha(a_2)^k)$  coincide on $Y$ and can only differ on $C_{> n_0}$ which has 
 measure less than $\epsilon/2$.
 Hence $d_u(w(\alpha(a_1),\alpha(a_2)^k),S)<\epsilon$ which implies that 
 $\overline{\alpha(\FF_r)}\geq  [\RR_T]$.

\begin{enumerate}[label=(\alph*)]
\setcounter{enumi}{1}
\item We claim that $\overline{\alpha(\FF_r)}\geq  [\RR_{\induced Y}]$.
\end{enumerate} 

Recall that $\alpha(a_2)= V(WU_2W^{-1}) (I \alpha_\infty(a_2))$ where $V$, $WU_2W^{-1}$ and $I\alpha_\infty(a_2)$ have pairwise disjoint support. Since $U_2$ extends $\varphi_2$, we get that $W\inv \alpha(a_2)W$ extends $\varphi_2$. By assumption $\alpha(a_3),\ldots,\alpha(a_r)$ extend $\varphi_3,\ldots\varphi_r$ respectively. Moreover $W\in[\RR_T]$, so the claim follows from Proposition~\ref{prop: the good pre-cycles}:
\[\overline{\alpha(\FF_r)}\geq \overline{\la[\RR_T],W^{-1}\alpha(a_2)W,\alpha(a_3),\ldots,\alpha(a_r)\ra}\geq [\RR_{\induced Y}].\]
 
 \begin{enumerate}[label=(\alph*)]
\setcounter{enumi}{2}
\item We claim that $\overline{\alpha(\FF_r)}\geq  [\RR]$.
\end{enumerate} 

This is a direct consequence of Corollary \ref{crl: erg+Amu1/2} granting
that $\alpha(\Gamma) Y=X$, which we will now show.

Clearly $\alpha(a_2) Y\supset \alpha(a_2)A=B=\cup_n B_n$. For every $n$ and 
$g\in G_n'\setminus \{\xi_n\}$, there exists $\gamma\in \Gamma$ of minimal 
length such that $\rho_n(\gamma)\xi_n=g$. Since $\rho_n(a_2)\xi_n=\xi_n$ and 
since $\alpha(a_2)_{\restriction C_n\setminus B_n}=\alpha_n(a_2)$ mimics the 
action of $\rho_n(a_2)$ on $G'_n\setminus \{\xi_n\}$, the minimality of the 
length of $\gamma$ implies that $\alpha(\gamma)B_n=B_n^g$. Since this is true 
for every $g\in G'_n$ we get $\alpha(\Gamma)Y\supset C_n$; and this holds for 
every $n$. We thus have $\alpha(\Gamma)Y=X$ as wanted.
\end{step}

\begin{step}[totipotency]
Consider a transitive action $\rho$ of $\FF_r$ on some infinite set. Let $H$ be 
a Schreier ball such that $\abs{H}\geq m_0$. Then by construction there exists 
$n$ such that $H=G_n\subseteq G'_n$. Also remark that the restriction of the 
Schreier graph of the action $\alpha$ to $\cup_{g\in G_n}B_n^g\subseteq C_n$ 
mimics the partial Schreier graph $H$. Since $\rho$ and $H$ are arbitrary, the 
Schreier 
graph of $\alpha$ contains every sufficiently large Schreier ball of every 
transitive action of $\FF_r$ and this finishes the proof of  the main theorem.
\qed
\end{step}

\begin{remark}\label{rem: continuum many IRS for R}
The subspace $Y\subseteq X$ chosen in Step~\ref{step: smaller set} of the above proof 
coincides with the subset where $\alpha(a_1)$ is aperiodic.
So the event ``no power of $a_1$ belongs to $\Lambda$'' has measure $\mu(Y)$ 
in the IRS $\Stabmap^{\alpha}_{*}\mu$ associated with $\alpha$.
Observe that the measure $\mu(Y)$ can be chosen to take any value from the non-empty interval 
$\left(\max\left\{\frac{\cost(R)-1}{r-1}, \frac{1}{2}\right\},1\right)$.
Recalling that the density of $\alpha(\Gamma)$ implies $\Stabmap^{\alpha}$ is essentially injective, then the following holds:
{\em Every ergodic p.m.p.\ equivalence relation $\RR$ of cost $<r$ can be realized (up to a null-set) by the 
action of $\FF_r\acts\Sub(\FF_r)$ for continuum many different totipotent IRS's of $\FF_r$.}
\end{remark}

\bibliographystyle{alpha}

 \bigskip
  {\footnotesize
  \noindent
  {A.C., \textsc{Institut für Algebra und Geometrie, KIT, 76128 Karlsruhe, Germany}}\par\nopagebreak \texttt{alessandro.carderi@kit.edu}

  \noindent
{D.G., \textsc{Université de Lyon, CNRS, ENS-Lyon, 
Unité de Mathématiques Pures et Appliquées,  69007 Lyon, France}}
\par\nopagebreak \texttt{damien.gaboriau@ens-lyon.fr}

  \noindent
{F.L.M., \textsc{Université de Paris, Sorbonne Université, CNRS,
    Institut de Mathématiques de Jussieu-Paris Rive Gauche,
    F-75013 Paris, France}}
\par\nopagebreak \texttt{{francois.le-maitre@imj-prg.fr}
}
}
\end{document}